\documentclass[10pt]{amsart}

\usepackage{amsfonts,amsmath,amsxtra,amsthm,amssymb,tikz,latexsym}
\usepackage{mathrsfs}

\usepackage[backref=page]{hyperref}

\usepackage{url, hypcap}
\hypersetup{colorlinks=true, citecolor=darkblue, linkcolor=darkblue}
\definecolor{darkblue}{rgb}{0.0,0,0.7} 

\newtheorem{theorem}{Theorem}[section]

\newtheorem{corollary}[theorem]{Corollary}
\newtheorem{proposition}[theorem]{Proposition}

\theoremstyle{definition}

\numberwithin{equation}{section}

\usepackage{graphics, graphicx, xcolor}
\definecolor{darkred}{rgb}{0.7,0,0} 

\usepackage[colorinlistoftodos]{todonotes}

\usepackage[vcentermath, enableskew]{youngtab}
\usepackage{lipsum,  verbatim}
\usepackage{mathtools}
\usepackage{txfonts}

\title{Generalized Staircase Tableaux: Symmetry and Applications}

\author[G.~Hawkes]{Graham Hawkes}
\address[G. Hawkes]{Department of Mathematics, UC Davis, One Shields Ave., Davis, CA 95616-8633, U.S.A.}
\email{hawkes@math.ucdavis.edu}


\begin{document}

\begin{abstract}
We define a number of related combinatorial objects, each of which possesses a surprising symmetry. We include several applications such as a combinatorial explanation for certain fixed points of the involution $\omega$ on the ring of symmetric functions, as well as a relationship between certain skew Schur functions and skew $Q$-Schur functions.  We give a $t$-deformation of these $Q$-Schur functions, and show that it is Schur positive, including a combinatorial description of the Schur coefficients.  A corollary of our results is the equality of skew $Q$-Schur functions:  $Q_{\lambda+\delta/\mu + \delta}=Q_{\lambda'+\delta/\mu' + \delta}$ for $\mu \subseteq \lambda$ and $\delta=(n,\ldots,1)$ for some $n > l(\lambda)$.
\end{abstract}

\maketitle

\section{Introduction}

The ring of symmetric functions, $\Lambda$, has a $\mathbb{Z}$-basis composed of Schur functions. Hence we can define an invertible linear operator $\omega$, by the formula $\omega(s_{\lambda})=s_{\lambda'}$.  We will call $f \in \Lambda$ a fixed point of $\omega$ if $\omega(f)=f$.  Clearly $s_{\lambda}$ is a fixed point for any self-conjugate partition $\lambda$.   Moreover, one can show that $\omega(s_{\lambda/\mu})=s_{\lambda'/\mu'}$, \cite{Stanley.1999} meaning that $s_{\lambda/\mu}$ is a fixed point for any self-conjugate partitions $\mu \subseteq \lambda$.  In particular $s_{\delta/\mu}$, where $\delta= (n,n-1,\ldots 1)$ is a fixed point for any self-conjugate $\mu \subseteq \delta$.  A priori there is little reason to expect that for any $\mu \subseteq \delta$ (not necessarily self-conjugate) $s_{\delta/\mu}$ would still be a fixed point. The fact that $\delta$ is self-conjugate and $\omega(s_{\lambda/\mu})=s_{\lambda'/\mu'}$ means that the statement above is equivalent to $s_{\delta/\mu}=s_{\delta/\mu'}$. However, this will be an immediate consequence of the symmetry of generalized staircase tableaux.

Besides being a fixed point of $\omega$, the function $s_{\delta/\mu}$ is interesting in its relation to shifted Schur functions.  For one, it is known, \cite{Ardilla.Serrano.2012} that $s_{\delta/\mu}$ is $P$-Schur positive.  We do not recreate this result here but do obtain the result that $s_{\delta/\mu}(x_1,x_1,x_2,x_2,\ldots)$ is $Q$-Schur positive.  (Note that $P$-Schur positivity is immediately guaranteed as a corollary of the result of \cite{Ardilla.Serrano.2012}, but not necessarily $Q$-Schur positivity.)  In particular, we derive that $s_{\delta/\mu}(x_1,x_1,x_2,x_2,\ldots)$ is equal to a certain skew $Q$-Schur function, from which the result follows by \cite{Worley.1984}, or \cite{Stembridge.1989}.

Our last application is deriving the equality of the skew $Q$-Schur functions $Q_{\lambda+\delta/\mu + \delta}=Q_{\lambda'+\delta/\mu' + \delta}$.  On the way to accomplishing this we define a $t$-deformation of $Q$-Schur functions and show that, in certain cases, it is symmetric and Schur positive, and give a combinatorial interpretation of the Schur coefficients.  This is accomplished with the help of a certain crystal structure introduced in \cite{Hawkes.Paramonov.Schilling.2017}.  We then prove that certain permutations to the order of the alphabet $1'<1<2'<2\cdots$,  which is typically used to define shifted semistandard tableaux (e.g. \cite{Serrano.2009}), may be made when the partition does not touch the diagonal.

Lastly, we note that the name generalized staircase tableaux (GST) is a bit deceptive, as we will define them for arbitrary shapes.  However, non-staircase GST are simply used to aid in our proofs, and no interesting results involve them.


\section{Definitions}
In what follows, we fix some $n$, and write $\delta=\delta=(n,n-1,\ldots,1)$.  Any partition denoted by $\mu$ which appears henceforth will be assumed to satisfy $\mu \subseteq \delta$.  Moreover, whenever $\lambda$ is mentioned we will assume also that $\mu \subseteq \lambda$ and that $l(\lambda)\leq n$.

A generalized staircase tableau (GST) of shape $\lambda/\mu$ and set $I \subseteq \mathbb{N}$ is a filling of the Young diagram $\lambda/\mu$ with natural numbers such that:

\begin{enumerate}
\item{The rows and columns are weakly increasing.}
\item{If $i \in I$ then each row has at most one $i$.}
\item{If $i \notin I$, then each column has at most one $i$.}
\end{enumerate}

Let $\textbf{G}(\lambda/\mu,I)$ denote the set of all GST of shape $\lambda/\mu$ and set $I$.  For instance,  $\textbf{G}(\lambda/\mu,\emptyset)$ is the set of skew semistandard Young tableau of shape ${\lambda/\mu}$, and the set $\textbf{G}(\lambda/\mu,2\mathbb{N}-1)$ is in in bijection with the set of skew shifted semi-standard tableaux of shape $(\delta+\lambda)/(\delta+\mu)$.

Suppose $T$ is a GST of shape $\lambda/\mu$ .   We will denote the box in row $i$ and column $j$ by $b_{ij}$.  We will denote the value inside $b_{ij}$ by $c(b_{ij})$ or by the ``content of $b_{ij}$."  For indexing purposes we will allow the coordinates of $i$ and $j$ to be any non-negative integers, although box $b_{ij}$ is always empty whenever $i=0$ or $j=0$.  We define the content of these border boxes to be $-\infty$. We also define an empty box inside of $\mu$ to have content $-\infty$.   On the other hand, an empty box outside of $\lambda$ (and with $i$ and $j$ positive) is defined to have content $\infty$. The weight of $T$, denoted by $wt(T)$ is defined to be the vector whose $i^{th}$ coordinate is equal to the number of $i$s appearing in $T$.

Suppose that $T$ is a GST of shape $\lambda/\mu$ and that $b_{ij} \in \mu$ such that $\mu-\{b_{ij}\}$ is a partition.  We define \emph{forward jdt} into $b_{ij}$ as follows.
\begin{enumerate}
\item{ Between the boxes  $b_{i(j+1)}$  and $b_{(i+1)j}$ select the box whose content is lesser.  If the contents are equal, then select $b_{i(j+1)}$ if $c(b_{i(j+1)}) =  c(b_{(i+1)j}) \in I$  and $b_{(i+1)j}$ otherwise.}
\item{If an empty box was selected, the algorithm terminates.  Otherwise move the content of the selected box into $b_{ij}$.}
\item{Re-index so that the newly emptied box is $b_{ij}$ and return to step 1.}
\end{enumerate}
Similarly, we define \emph{reverse jdt} into a box $b_{ij} \notin \lambda$ such that $\lambda \cup \{b_{ij}\}$ is a partition as follows:
\begin{enumerate}
\item{ Between the boxes  $b_{i(j-1)}$  and $b_{(i-1)j}$ select the box whose content is greater.  If the contents are equal, then select $b_{i(j-1)}$ if $c(b_{i(j-1)}) =  c(b_{(i-1)j}) \in I$  and $b_{(i-1)j}$ otherwise.}
\item{If an empty box was selected, the algorithm terminates.  Otherwise move the content of the selected box into $b_{ij}$.}
\item{Re-index so that the newly emptied box is $b_{ij}$ and return to step 1.}
\end{enumerate}

Note that in both cases a valid GST is returned.  Moreover, our jdt satisfies the familiar properties of classic jdt:

\begin{enumerate}
\item [J1]  {If $T'$ is obtained from $T$ by \emph{forward jdt} into $b_{ij}$, and $b_{i'j'}$ is the last box to be emptied, then $T$ can be obtained from $T'$ by \emph{reverse jdt} into $b_{i'j'}$.}
\item[J2]{If $T'$ is obtained from $T$ by \emph{reverse jdt} into $b_{ij}$, and $b_{i'j'}$ is the last box to be emptied, then $T$ can be obtained from $T'$ by \emph{forward jdt} into $b_{i'j'}$.}
\item[J3]{If $i \geq k$ and $j<l$ and it is possible to \emph{forward jdt} into $b_{kl}$ and then \emph{forward jdt} into $b_{ij}$, and the boxes emptied by doing this are (in order) $b_{k'l'}$  and $b_{i'j'}$   then $i' \geq k'$ and $j'<l'$.}

\item[J4]{If $i < k$ and $j\geq l$ and it is possible to \emph{forward jdt} into $b_{kl}$ and then \emph{forward jdt} into $b_{ij}$, and the boxes emptied by doing this are (in order) $b_{k'l'}$  and $b_{i'j'}$   then $i' < k'$ and $j'\geq l'$.}
\item[J5]{If $i \geq k$ and $j<l$ and it is possible to \emph{reverse jdt} into $b_{ij}$ and then \emph{reverse jdt} into $b_{kl}$, and the boxes emptied by doing this are (in order) $b_{i'j'}$  and $b_{k'l'}$   then $i' \geq k'$ and $j'<l'$.}
\item[J6]{If $i < k$ and $j\geq l$ and it is possible to \emph{reverse jdt} into $b_{ij}$ and then \emph{reverse jdt} into $b_{kl}$, and the boxes emptied by doing this are (in order) $b_{i'j'}$  and $b_{k'l'}$   then $i' < k'$ and $j'\geq l'$.}

\end{enumerate}

\section{Results}

\begin{theorem}\label{main}
 If $I$ and $I'$ are any subsets of the natural numbers, there is a weight preserving bijection from $\textbf{G}(\delta/\mu,I)$ to $\textbf{G}(\delta/\mu,I')$.
\end{theorem}

\begin{proof}
It suffices to show that for any $I$ and any $i \notin I$ there is a weight preserving bijection from  $\textbf{G}(\delta/\mu,I)$ to  $\textbf{G}(\delta/\mu,I \cup i)$.  Let $\mu \subseteq \nu \subseteq \lambda$, be the partition consisting of all boxes of $\mu$ and all boxes of $\lambda$ with content less than $i$.  It will suffice to find a weight-preserving bijection from $\textbf{G}(\delta/\nu, I_{\geq i})$ to  $\textbf{G}(\delta/\nu,I_{\geq i}\cup i)$ or equivalently from $\textbf{G}(\delta/\nu, J)$ to  $\textbf{G}(\delta/\nu,J \cup 1)$ where $J=[(-i+1)+I]_{\geq 1}$.

We therefore assume that in the statement of the theorem, $i=1 \notin I$, and  $I'=I\cup 1$. Let $T \in \textbf{G}(\delta/\mu,I)$.  First erase all of the 1s which appear in $T$.  The result is a horizontal strip of empty boxes on the inside of the tableau.  Now, \emph{forward jdt} into each one of these empty boxes starting with the rightmost and moving left.  By property J3, the boxes which are emptied along the outside of the tableau will form a horizontal strip, and they will be emptied from right to left.  However, since $\delta$ is the staircase shape, this strip is in fact also a vertical strip.  Now, \emph{reverse jdt} into each of the boxes of this vertical strip, starting with the highest and moving down.  By property J6, the boxes emptied along the inside of the tableau will form a vertical strip, and they will be emptied starting with the highest and moving down.  Put a 1 into each of the newly emptied boxes.  This produces a tableau in $\textbf{G}(\delta/\mu,I \cup 1)$ which we define to be $\phi(T)$.  Now, given any $T \in \textbf{G}(\delta/\mu,I \cup 1)$ define $\phi^{-1}(T)=\phi(T^t)^t$, where the superscript $t$ stands for row/column transposition.  It is not hard to check that $\phi^{-1}(T) \in \textbf{G}(\delta/\mu,I)$ and that both $\phi \circ \phi^{-1}=Id$ and $\phi^{-1} \circ \phi=Id$. \\

\end{proof}

\begin{theorem}
 $s_{\delta/\mu}$ is a fixed point of the involution $\omega$. 
\end{theorem}

\begin{proof}
By the comments in the introduction, this is equivalent to showing that $s_{\delta /\mu}=s_{\delta /\mu'}$.  But this equality is equivalent to the equality:
\begin{eqnarray*}
\sum_{T \in \mathbf{G}(\delta /\mu,\emptyset)}\mathbf{x}^{wt(T)}=\sum_{T \in \mathbf{G}(\delta /\mu, \mathbb{N})}\mathbf{x}^{wt(T)},
\end{eqnarray*}
which is true because $\mathbf{G}(\delta /\mu,\emptyset)$ and $\mathbf{G}(\delta /\mu,\mathbb{N})$ are in weight preserving bijection.

\end{proof}

  Our next application relates Schur and $Q$-Schur functions of certain shapes:  For our purposes we will define a $Q$-tableau  to be a filling of the shape  $\lambda/\mu$ using letters from the ordered alphabet $1'<1<2'<2\ldots$ such that:

\begin{enumerate}
\item{The rows and columns are weakly increasing.}
\item{No primed number appears more than once in any row.}
\item{No unprimed number appears more than once in any column.}
\end{enumerate}

 We define the reading word of a $Q$-tableau to be the word obtained by reading the  primed entries in $T$ down columns from right to left and then reading the the unprimed entries left to right across rows, starting with the lowest row and working up.  We consider this word as a word in the alphabet $\{1,2,3,\ldots\}$ by ignoring the primes which appear above the entries at the beginning of the word.  This is the same definition as for the reading word of ``primed tableaux" in \cite{Hawkes.Paramonov.Schilling.2017}, except that here, our tableaux are not shifted.
 
We define a  function:
\begin{eqnarray*}
 Q_{\lambda/\mu }^{tr}=\sum_T \mathbf{x}^{wt(T)}t^{P(T)}r^{U(T)},
\end{eqnarray*}
where the sum is over all $Q$-tableaux of shape $\lambda/\mu$, where $wt(T)$ is the vector whose $i^{th}$ coordinate counts the number of times either $i$ or $i'$ appears in $T$, and where $P(T)$ (resp. $U(T)$) counts the number of times a primed (resp. unprimed) entry appears in $T$.  Notice that, by definition,  $Q^{tr}_{\lambda/\mu}$ at $t=1=r$ is the $Q$-Schur function $Q_{\lambda+\delta/\mu + \delta}$.

\begin{theorem}
\begin{eqnarray*}
Q_{\lambda/\mu }^{tr}=\sum_{k}\Bigg(\sum_{\nu} c_{\lambda/\mu}^{\nu,k}s_{\nu}\Bigg)t^kr^{|\lambda|-|\mu|-k}
\end{eqnarray*}
Where $c_{\lambda/\mu}^{\nu,k}$ is the number of $Q$-tableau of shape ${\lambda/\mu}$ and weight $\nu$ which have exactly $k$ of their entries primed, and whose reading word is Yamanouchi.
\end{theorem}

\begin{proof}
The crystal operators on primed tableau given in \cite{Hawkes.Paramonov.Schilling.2017} induce crystal operators on skew primed tableaux in the natural way.   Notice that the set of skewed primed tableaux of shape ${\lambda+\delta/\mu + \delta}$ is cannonically equivalent to the set of all $Q$-tableaux of shape $\lambda/\mu$, and so we obtain a crystal structure on the latter.  Moreover, when the set of $Q$-tableaux of shape $\lambda/\mu$ inherits this structure, the highest weight elements of this crystal will be those $Q$-tableaux whose reading word is Yamanouchi. This follows directly from the description of highest weight primed tableaux given in \cite{Hawkes.Paramonov.Schilling.2017}. In order to prove the theorem, it remains to show that $P(T)$ is constant on connected components of the induced crystal on $Q$-tableaux. However, one may check that the crystal operator $f_i$ in \cite{Hawkes.Paramonov.Schilling.2017} preserves the number of primes in a given primed tableau whenever this tableau has no $is$ or $(i+1)$s on the diagonal.  However, note that we are associating $Q$-tableaux of shape $\lambda/\mu$ to primed tableaux of shifted skew shape  ${\lambda+\delta/\mu + \delta}$, and that the latter shape has no boxes on the diagonal. Thus, it is the case that for all $i$, we are always applying the operator $f_i$ to a (skew) primed tableaux with no $i$s or $(i+1)$s on the diagonal.  Thus, the induced operators on $Q$-tableaux also preserve the number of primes in a given $Q$-tableau.
\end{proof}

\begin{theorem}
\begin{eqnarray*}
s_{\delta/\mu }(tx_1,rx_1,tx_2,rx_2,\ldots)=Q_{\delta/\mu }^{tr}=\sum_{k}\Bigg(\sum_{\nu} c_{\delta/\mu}^{\nu,k}s_{\nu}\Bigg)t^kr^{|\delta|-|\mu|-k}
\end{eqnarray*}
In particular, $s_{\delta/\mu }(x_1,x_1,x_2,x_2,\ldots)$ is $Q$-Schur positive (since it is equal to  the skew $Q$-Schur function $Q_{\lambda+\delta/\mu + \delta}$ which is $Q$-Schur positive by \cite{Worley.1984} or \cite{Stembridge.1989}), and we have that $s_{\delta/\mu }(tx_1,x_1,tx_2,x_2,\ldots)$ is Schur positive.
\end{theorem}

\begin{proof}

\begin{eqnarray*}
s_{\delta/\mu }(tx_1,rx_1,tx_2,rx_2,\ldots)=\sum_T \mathbf{x}^{wt(T)}t^{P(T)}r^{|\delta|-|\mu|-P(T)},
\end{eqnarray*}
where we claim the sum can be taken over any of the following:
\begin{enumerate}
\item{ Over all SSYT of shape $\delta/\mu$, where $wt(T)$ is the vector whose $i^{th}$ coordinate counts the number of times either $2i-1$ or $2i$ appears in $T$, and where $P(T)$ counts the number of times an odd entry appears in $T$.}
\item{ Over $\mathbf{G}(\delta/\mu,\emptyset)$, where $wt(T)$ is the vector whose $i^{th}$ coordinate counts the number of times either $2i-1$ or $2i$ appears in $T$, and where $P(T)$ counts the number of times an odd entry appears in $T$.}
\item{ Over $\mathbf{G}(\delta/\mu,2\mathbb{N}-1)$, where $wt(T)$ is the vector whose $i^{th}$ coordinate counts the number of times either $2i-1$ or $2i$ appears in $T$, and where $P(T)$ counts the number of times an odd entry appears in $T$.}
\item {Over all $Q$-tableaux of shape $\delta/\mu$, where $wt(T)$ is the vector whose $i^{th}$ coordinate counts the number of times either $i$ or $i'$ appears in $T$, and where $P(T)$ counts the number of times a primed entry appears in $T$.}
\end{enumerate}

(1) is true by definition.  (1)$\implies$ (2) by the definition of GST.  (2) $\implies$ (3) by \ref{main}.  (3) $\implies$ (4) by relabeling the alphabet, and (4) corresponds to the statement in the theorem.

\end{proof}

\begin{corollary}
$Q_{\delta/\mu }^{tr}$ is symmetric in $t$ and $r$, and $Q_{\delta/\mu }^{tr}=Q_{\delta/\mu' }^{tr}$.  In particular, we have the equality of skew $Q$-Schur functions $Q_{\delta+\delta/\mu + \delta}=Q_{\delta+\delta/\mu' + \delta}$.

\end{corollary}

In fact, more generally we have:

\begin{proposition}\label{tr}
$Q_{\lambda/\mu}^{tr}(\mathbf{x};t,r)=Q_{\lambda'/\mu'}^{tr}(\mathbf{x};r,t)$.
\end{proposition}

Before proving this, we introduce a generalization of $Q$-tableau.  Let $I \subseteq \mathbb{N}$ and define  the total order $\leq_I$ on the alphabet $\{1',1,2',2,\ldots\}$ by
\begin{enumerate}
\item{ If $i<j$ then $i <_I j$, $i<_I j'$, $i' <_I j$, $i'<_I j'$}
\item{If $i \in I$ then $i<_I i'$}
\item{If $i \notin I$ then $i'<_I i$}
\end{enumerate}

We define a generalized $Q$-tableau of shape $\lambda/\mu$ and set $I$ to be a filling of this shape using $\{1',1,2',2,\ldots\}$  such that:

\begin{enumerate}
\item{The rows and columns are weakly increasing under $\leq_I$.}
\item{No primed number appears more than once in any row.}
\item{No unprimed number appears more than once in any column.}
\end{enumerate}
The set of all such tableaux is denoted $\mathbf{Q}(\lambda/\mu,I)$.

\begin{theorem}\label{rib}
For any subsets of the natural numbers, $I$, and $I'$, there is a bijection from $\mathbf{Q}(\lambda/\mu,I)$ to $\mathbf{Q}(\lambda/\mu,I')$ which preserves $wt(T)$ and $P(T)$.
\end{theorem}
\begin{proof}
It suffices to suppose that $I'=I \cup i$ for some $i \notin I$.  Let  $T \in \mathbf{Q}(\lambda/\mu,I)$ and define $\psi(T)$ as follows.  First, write down $T$. Notice that the $is$ and $i's$ in $T$ form a set of connected ribbons.  Within each of these connected ribbons, cycle every entry one position: to the left if  the box to its left is in the ribbon, downwards if the box below it is in the ribbon, or, if neither is the case, i.e., it is at the bottom left end of the ribbon, move it to the upper right end of the ribbon.   $\psi^{-1}$ is defined similarly, but by cycling the other direction.
\end{proof}

We can now prove \ref{tr}.

\begin{proof}
We seek a weight preserving bijection from $\mathbf{Q}(\lambda'/\mu',\emptyset)$ to $\mathbf{Q}(\lambda/\mu,\emptyset)$ which interchanges $P(T)$ and $U(T)$. Let $T \in \mathbf{Q}(\lambda'/\mu',\emptyset)$.  Tranpose $T$ and then prime the unprimed elements and unprime the primed elements. This gives a weight preserving bijection from $\mathbf{Q}(\lambda'/\mu',\emptyset)$ to $\mathbf{Q}(\lambda/\mu,\mathbb{N})$ which interchanges $P(T)$ and $U(T)$, and by \ref{rib}, this is sufficient.
\end{proof}

\begin{corollary}
We have the equality of skew $Q$-Schur functions $Q_{\lambda+\delta/\mu + \delta}=Q_{\lambda'+\delta/\mu' + \delta}$.
\end{corollary}

(Here we make the additional assumption that $\lambda_1 \leq n$.)

\nocite{*}
\bibliographystyle{alpha}
\bibliography{stare}

\end{document}